\documentclass[12pt,oneside]{amsart}
\usepackage{amssymb, amsmath, amsthm}

\usepackage{epsfig}
\usepackage{epstopdf}
\usepackage{graphicx}
\usepackage{color}
\usepackage{mathptmx}
\usepackage{enumerate}

\usepackage{pinlabel}

\theoremstyle{plain}
\newtheorem{theorem}{Theorem}[section]

\newtheorem*{theorem*}{Theorem}
\newtheorem*{theorem-DisjtSS}{Theorem \ref{Thm: Disjt SS}}
\newtheorem*{theorem-EssentialTorus}{Theorem \ref{Thm: Essential Torus}}
\newtheorem*{cor-ScharlWu}{Corollary \ref{Cor-ScharlWu}}
\newtheorem*{corollary-MSC 2}{Corollary \ref{Cor: MSC 2}}
\newtheorem*{theorem-Main A}{Theorem \ref{Thm: Main A}}
\newtheorem*{theorem-Main B}{Theorem \ref{Thm: Main Thm B}}
\newtheorem*{Cor-Unknotting}{Theorem \ref{Cor: Prime Unknotting 1}}
\newtheorem*{Cor-genusbandsum}{Theorem \ref{Thm: Genus superadd}}
\newtheorem*{Cor-bandsumscc}{Corollary \ref{Cor: Band Sums CC}}

\newtheorem{proposition}[theorem]{Proposition}
\newtheorem{corollary}[theorem]{Corollary}
\newtheorem{lemma}[theorem]{Lemma}

\theoremstyle{definition}

\newtheorem*{remark}{Remark}

\newtheorem*{conjecture}{Conjecture}

\newcommand{\nil}{\varnothing}
\newcommand{\tild}{\widetilde}
\newcommand{\wihat}{\widehat}
\newcommand{\defn}[1]{\textbf{#1}}
\newcommand{\boundary}{\partial}

\newcommand{\mc}[1]{\mathcal{#1}}
\newcommand{\ob}[1]{\overline{#1}}
\newcommand{\genus}{\operatorname{genus}} 
\newcommand{\inter}[1]{\mathring{#1}}

      \makeatletter
      \def\@setcopyright{}
      \def\serieslogo@{}
      \makeatother
      
\begin{document}

   \title[Comparing 2-handle additions]{Comparing 2-handle additions to a genus 2 boundary component}
   \author{Scott A Taylor}
   \email{sataylor@colby.edu}
   \thanks{}
\begin{abstract}
We prove that knots obtained by attaching a band to a split link satisfy the cabling conjecture. We also give new proofs that unknotting number one knots are prime and that genus is superadditive under band sum. Additionally, we prove a collection of results comparing two 2-handle additions to a genus two boundary component of a compact, orientable 3-manifold. These results give a near complete solution to a conjecture of Scharlemann and provide evidence for a conjecture of Scharlemann and Wu. The proofs make use of a new theorem concerning the effects of attaching a 2-handle to a suture in the boundary of a sutured manifold. 
\end{abstract}
\maketitle
\date{\today}

\section{Introduction}
Sutured manifold theory, since its creation, has been used for comparing the results of two Dehn surgeries on a knot in a 3-manifold. Probably the best known result along these lines is Gabai's theorem \cite[Corollary 2.4]{G2} that at most one Dehn filling on a torus boundary component of a compact, orientable 3-manifold, can decrease Thurston norm. In this paper we use a new theorem from \cite{T3} concerning 2-handle addition to a sutured manifold to compare the results of two 2-handle additions to a genus two boundary component of a compact, orientable 3-manifold. We begin by stating our results, none of the statements of which involve sutured manifolds. In Section \ref{Sutured Manifolds} we review the definition of ``sutured manifold'' and state the theorem (Theorem \ref{Thm: Main Theorem}) from \cite{T3} that is fundamental to the present work.

Throughout the paper we adopt the following conventions and notation. An \defn{essential surface} in a 3--manifold is one that is incompressible, not a 2-sphere bounding a ball, and not boundary parallel. The regular neighborhood of a space $X$ is denoted $\eta(X)$ and the interior of $X$ is denoted $\inter{X}$. $N$ is a compact, orientable 3--manifold with a boundary component $F$ containing an essential simple closed curve $b$. Let $N[b]$ denote the result of attaching a 2--handle to $N$ along a regular neighborhood of $b$ in $F$ and let $\beta \subset N[b]$ denote the properly embedded arc that is the cocore of the 2--handle attached to $b$. Let $\boundary_1 N = \boundary_1 N[b]$ be the union of the components of $\boundary N - F$ having genus at least two. Let $\boundary_0 N[b] = \boundary N[b] - (\boundary N - F)$. Thus $\boundary N - (\boundary_0 N \cup \boundary_1 N \cup F)$ and $\boundary N[b] - (\boundary_0 N[b] \cup \boundary_1 N[b])$ are empty or consist of tori.

\subsection{Non-simple 2-handle additions}

A 3-manifold is \defn{simple} if it contains no essential, sphere, disc, annulus or torus. 

If $\boundary N$ is compressible but $\boundary N - b$ is incompressible, Jaco \cite{J} showed that $\boundary N[b]$ is incompressible. Eudave-Mu\~noz \cite{EM2} gave upper bounds on the minimal number of times that an essential annulus or torus intersects $\beta$. Scharlemann \cite{S2} showed that for any non-zero homology class $y \in H_2(N[b],\boundary N[b])$ there is a Thurston-norm minimizing surface representing $y$ in $N[b]$ disjoint from $\beta$. Since adding a 2-handle to $\boundary N$ along the boundary of an essential disc creates a reducible 3--manifold, each of these results can be viewed as a comparision theorem between reducible 2-handle addition and either non-simple 2-handle addition or Thurston-norm reducing 2-handle addition. Of course, \textit{a priori} there may be reducible 2-handle additions that do not arise from attaching a 2-handle along the boundary of the disc. Also note that the aforementioned results concern 2-handle attachment to a non-simple 3-manifold. 

Some work has been done comparing non-simple 2-handle attachments to the boundary of a simple 3-manifold. The curve along which such a 2-handle is attached is called a \defn{degenerating} curve.  Here is what was known, prior to this paper, concerning 2--handle attachments along degenerating curves $a$ and $b$. None of these results use sutured manifold theory. We let $|a \cap b|$ denote the minimal intersection number between $a$ and $b$.

\begin{theorem*}
Suppose that $N$ is a simple 3--manifold and that $F \subset \boundary N$ is a component of genus at least 2 containing essential simple closed curves $a$ and $b$. Then
\begin{itemize}
\item If $N[b]$ is reducible and $N[a]$ is boundary-reducible then either $|a \cap b| = 0$ or $a$ and $b$ can be isotoped to lie in a common once-punctured torus in $\boundary N$ \cite[Theorem 4.2]{SW}.
\item If $N[a]$ and $N[b]$ are both reducible, then $|a \cap b| \leq 4$ \cite{ZQL}.
\item If $\genus(F) = 2$, if $a$ and $b$ are both separating, and if $N[a]$ and $N[b]$ are both boundary-reducible, then $|a \cap b| = 0$ \cite{LQZ}.
\end{itemize}
\end{theorem*}

A degenerating curve $a \subset \boundary N$ is called \defn{basic} if $a$ is separating or if there is no degenerating separating curve $a^*$ bounding in $\boundary N$ a once-punctured torus containing $b$. Scharlemann and Wu conjecture:

\begin{conjecture}[{\cite[Conjecture 2]{SW}}]
If $N$ is simple and if $a$ and $b$ are basic degenerating curves on $\boundary N$ then $|a \cap b| \leq 5$.
\end{conjecture}

The next theorem gives some additional evidence for their conjecture. Its proof is logically independent from the prior results on 2-handle addition.

\begin{cor-ScharlWu}
Suppose that $N$ is a compact orientable simple 3-manifold and that $F \subset \boundary N$ is a genus two boundary component. Let $a,b \subset F$ be essential simple closed curves with $|a \cap b| \geq 1$. Assume that $b$ is separating and that $N[b]$ is reducible. Then $N[a]$ is irreducible and boundary-irreducible and if $a$ is a degenerating curve, then $a$ is non-separating and $|a \cap b| = 2$. Furthermore, if $N[a]$ contains an essential annulus, then it contains one with boundary disjoint from $b \cap \boundary N[a]$.
\end{cor-ScharlWu}

\subsection{Refilling Meridians and Boring}
If $N$ can be embedded in a 3--manifold $M$ so that the genus 2 component $F \subset \boundary N$ bounds a handlebody $W$ in $M - N$ and if the curves $a$ and $b$ bound discs $A$ and $B$ in $W$ we say that the 3--manifolds $N[a]$ and $N[b]$ are obtained by \defn{refilling} the meridians $A$ and $B$ respectively. We denote the core or cores of the solid torus or tori obtained from boundary reducing $W$ using $A$ by $L_a$ and similarly define the knot or link $L_b$. If $A$ and $B$ cannot be isotoped to be disjoint we say that $L_a$ and $L_b$ are related by \defn{boring} and that $L_a$ is obtained by boring $L_b$ (and vice versa). The arcs $\alpha$ and $\beta$ that are the cocores of $\eta(A)$ and $\eta(B)$ in $N[a]$ and $N[b]$ are called \defn{boring arcs}.

The paper \cite{T1} proves several theorems about knots and links obtained by boring a split link. In most ways, the results of this paper supercede that paper; however, the methods of that paper are of interest in their own right and are reasonably effective at studying homology classes $y \in H_2(N[b],\boundary N[b])$ such that the projection of $\boundary y$ to $\boundary_0 N[b]$ is zero. The techniques of this paper are not very useful for drawing conclusions about such classes when $b \subset F$ is separating.

Scharlemann \cite{MSc5} conjectured that if $M$, $N$, and $W$ satisfy certain rather mild hypotheses than at least one of $N[a]$ or $N[b]$ is irreducible or boundary-irreducible. In \cite{T1}, considerable progress was made on the conjecture (actually a minor variation of it). The only significant case remaining was to prove that if $M$, $N$, and $W$ satisfy certain mild conditions then it is impossible for both $N[a]$ and $N[b]$ to be solid tori. One corollary of Theorem \ref{Thm: MSC 2} is:
\begin{corollary-MSC 2}
Assume that $W$ is a genus 2 handlebody embedded in a compact, orientable 3--manifold $M$ such that $M$ contains no lens space connected summands, any pair of curves in $\boundary M$ that compress in $M$ are isotopic in $M$, and $N = M - \inter{W}$ is irreducible and boundary-irreducible. If $a$ and $b$ are essential closed curves on $\boundary W$ bounding discs in $W$ that cannot be isotoped to be disjoint then one of the following occurs:
\begin{enumerate}
\item One of $N[a]$ and $N[b]$ is irreducible and boundary-irreducible
\item There exists an essential annulus $A \subset N$ such that one component of $\boundary A$ lies on a component of $\boundary M$ and the other lies on $\boundary W$, is disjoint from $a$ or $b$ and bounds a disc in $W$.
\end{enumerate}
\end{corollary-MSC 2}
Thus, the only case of Scharlemann's conjecture that remains unanswered is when there exists an essential annulus in $N$ with one boundary component on $\boundary M$ and the other on either $a$ or $b$.

\subsection{Rational Tangle Replacement}

A special case of the operation of boring is rational tangle replacement. Consider a 2-tangle $(D,\tau)$ embedded in $S^3$ and consider rational tangles $(D',r_a)$ and $(D',r_b)$ such that $D' = S^3 - \inter{D}$ and $L_a = r_a \cup \tau$ and $L_b = r_b \cup \tau$ are each knots or 2-component links. We say that $L_a$ and $L_b$ are obtained from each other by \defn{rational tangle replacement}. The genus two handlebody $W = D' \cup \eta(\tau)$ contains discs $A$ and $B$ such that reducing $W$ along $A$ creates a regular neighborhood of $L_a$ and reducing $W$ along $B$ creates a regular neighborhood of $L_b$. (The discs $A$ and $B$ are the discs in $D'$ separating the strands of $r_a$ and $r_b$.) Assume that $A$ and $B$ have been properly isotoped in $D' - \boundary \tau$ to minimize $d = |A \cap B|$. The number $d$ is called the \defn{distance} of the rational tangle replacement. It coincides with the usual notion of distance between the rational tangles $(D',r_a)$ and $(D',r_b)$. Letting $a = \boundary A$ and $b = \boundary B$, we also have $d = |a \cap b|/2$. As usual, let $\alpha$ and $\beta$ be the cocores of the 2-handles with cores $A$ and $B$ respectively. 

Performing a crossing change on a knot or 2-component link is a rational tangle replacement of distance 2 and attaching a (possibly twisted) band to a knot or 2-component link is a rational tangle replacement of distance 1. If a crossing change on a non-trivial knot $K$ converts it into the unknot then $K$ has \defn{unknotting number one} and if $K$ is a knot formed by attaching a band to the components of a split link then $K$ is the \defn{band sum} of the components of the split link. 

For somewhat technical reasons, rational tangle replacement is particularly amenable to study by the techniques of this paper. We give three examples. The first is a new proof of an old theorem \cite{S1} of Scharlemann. A different sutured manifold theory proof of this result was given in \cite{ST}.

\begin{Cor-Unknotting}
Unknotting number one knots are prime.
\end{Cor-Unknotting}

The second example is a new proof of a result of Gabai and Scharlemann (proved in both cases using sutured manifold theory):

\begin{Cor-genusbandsum}
If a knot $K$ is the band sum of knots $K_1$ and $K_2$, then the genus of $K$ is at least the sum of the genera of $K_1$ and $K_2$.
\end{Cor-genusbandsum}

The advantage of our proofs is that the sutured manifold theory developed in \cite{T2} and here has a number of other applications and can give more-or-less unified proofs of these results. The previous proofs were all idiosyncratic, particulary in the combinatorics. The basis of each of sutured manifold theory proofs of these results centers on showing that one of the arcs defining the rational tangle replacement can be made disjoint from a minimal genus Seifert surface. The technology of this paper allows a fairly streamlined approach to proving such a result.

Our final application to rational tangle replacment shows that band sums satisfy the cabling conjecture of Gonz\'alez-Acu\~na and Short \cite{GS}. In fact, no surgery on a band sum produces a reducible 3-manifold. In addition to relying on the work of \cite{T3} its proof also relies on Scharlemann's work in \cite{S4}.

\begin{Cor-bandsumscc}
If a knot $K$ is formed by attaching a band to both components of a split link, then $K$ is not a cable knot and no Dehn surgery on $K$ produces a reducible 3--manifold.
\end{Cor-bandsumscc}

The proof is given in Section \ref{Band Sum}.

\subsection{Acknowledgements} Versions of the results of sections \ref{Degenerating}, \ref{Refilling}, and \ref{RTR} appear in \cite{T1}.  I am grateful to Marty Scharlemann for helpful conversations. Portions of the work in this paper were supported by a grant from the National Science Foundation.

\section{Sutured Manifolds}\label{Sutured Manifolds}
A sutured manifold $(N,\gamma)$ consists of a compact, orientable 3-manifold $N$ and a collection of oriented simple closed curves (called \defn{sutures}) $\gamma \subset \boundary N$,  such that:
\begin{itemize}
\item $\boundary N - \inter{\eta}(\gamma)$ consists of two (possibly disconnected) surfaces $R_- = R_-(\gamma)$ and $R_+ = R_+(\gamma)$ such that each component of $\gamma$ is adjacent to both $R_-$ and $R_+$. Let $R(\gamma) = R_-(\gamma) \cup R_+(\gamma)$. 
\item $R_-(\gamma)$ has an inward normal orientation and $R_+(\gamma)$ has an outward normal orientation.
\item Each component of $\boundary R(\gamma)$ has orientation induced from that of $R(\gamma)$ that coincides with the orientation of the adjacent component of $\gamma$.
\end{itemize}

The \defn{Thurston norm} of a connected surface $S$ is 
\[
x(S) = \max(-\chi(S),0).
\]
The Thurston norm of a disconnected surface is the sum of the Thurston norms of its components. An oriented surface $S$ is \defn{taut} if $S$ is incompressible  and if out of all surfaces with the same boundary as $S$ and homologous to $S$ in $H_2(M,\boundary S)$, the Thurston norm of $S$ is minimal. 

A sutured manifold $(M,\gamma)$ is \defn{taut} if $M$ is irreducible and if $R_-$ and $R_+$ are taut.

If $b \subset \boundary N$ is an simple closed curve, and if $Q \subset N$ is a properly embedded surface, then a \defn{$b$-boundary compressing disc} is a disc $D$ in $N$ with interior disjoint from $Q$ such that $\boundary D$ is the endpoint union of an arc embedded in $b$ and an arc embedded in $Q$. 

This paper gives a number of varied applications of the following theorem\footnote{In the statement below we have modified the statement of the theorem by specializing to the case when $F$ has genus 2 and $|\gamma| \in \{1,3\}$.} that was proven in \cite{T3}. 

\begin{theorem}\label{Thm: Main Theorem}
Suppose that $(N,\gamma)$ is a taut sutured manifold with a genus two boundary component $F$. Assume that $\gamma \cap F$ consists either of one separating curve or three non-separating curves. Let $b$ be a component of $\gamma \cap F$ and let $Q$ be a compact, orientable surface in $N$ such that:
\begin{itemize}
\item $|Q \cap b| \geq 1$
\item $\boundary Q$ intersects $\gamma \cap F$ minimally
\item No component of $Q$ is a sphere or a disc disjoint from $\gamma$.
\end{itemize}

Let $\beta$ be the cocore in $N[b]$ of a 2-handle attached along $b$. Then one of the following is true:
\begin{enumerate}
\item $Q$ has a compressing or $b$-boundary compressing disc.
\item $(N[b],\beta) = (M'_0,\beta'_0) \# (M'_1,\beta'_1)$ where $M'_1$ is a lens space and $\beta'_1$ is a core of a genus one Heegaard splitting of $M'_1$.
\item $(N[b],\gamma-b)$ is taut. The arc $\beta$ can be properly isotoped  to be embedded on a branched surface $B(\mc{H})$ associated to a taut sutured manifold hierarchy $\mc{H}$ for $N[b]$. There is also a proper isotopy of $\beta$ in $N[b]$ to an arc disjoint from the first decomposing surface in $\mc{H}$. That first decomposing surface can be taken to represent $\pm y$ for any given non-zero $y \in H_2(N[b],\boundary N[b])$.
\item 
\[
-2\chi(Q) + |Q \cap \gamma| \geq 2|Q \cap b|.
\]
\end{enumerate}
\end{theorem}

\begin{remark}
The taut sutured manifold hierarchy $\mc{H}$ constructed in conclusion (3) is constructed so that the first surface (which represents $y$) is ``conditioned''. 
A compact, orientable surface $S$ in a sutured manifold $(N,\gamma)$ is \defn{conditioned} if all components of $S \cap \eta(\gamma)$ in each component of $\eta(\gamma)$ are coherently oriented arcs or circles intersecting $\gamma$ minimally and if no collection of components of $\boundary S \cap R(\gamma)$ is trivial in $H_1(R(\gamma), \boundary R(\gamma))$. Thus, if $N[b] \subset M$ is formed by refilling a meridian of a genus 2 handlebody $W$ with $\boundary W = F$, and if $L_b$ is null-homologous in $M$, then $y$ can be chosen so that the first surface of $\mc{H}$ is a minimal genus Seifert surface for $L_b$. This is discussed in greater detail in Section 10 of \cite{T3}.
\end{remark}

To apply Theorem \ref{Thm: Main Theorem} effectively, we need to find surfaces $Q$ having no compressing or $b$-boundary compressing discs in $N$. That is the topic of the next section.

\section{Suitably Embedded Surfaces}

Let $N$ be a compact, orientable 3-manifold with a genus 2 boundary component $F$. Suppose that $a,b \subset F$ are essential simple closed curves. In this section, we show how to use a given essential surface $\ob{R}$ in $N[a]$ to produce an incompressible and $b$-boundary incompressible surface $Q \subset N$. 

First, we analyze ways in which an essential surface $\ob{R} \subset N[a]$ intersects $N$. If $a$ is non-separating, there are multiple ways to obtain a manifold homeomorphic to $N[a]$. Certainly attaching a 2--handle to $a$ is one such way, but there are other ways. If $a^* \subset F$ is a simple closed curve disjoint from $a$ and bounding a once-punctured torus containing $a$, then attaching 2--handles to both $a^*$ and $a$ creates a manifold with a spherical boundary component. Filling in that sphere with a 3--ball creates a manifold homeomorphic to $N[a]$. We will often think of $N[a]$ as obtained in this fashion. If $a$ is separating, define $a^* = \nil$. 

Let $\ob{Q} \subset N[a]$ be an embedded surface and let $Q = \ob{Q} \cap N$. The surface $\ob{Q}$ is \defn{suitably embedded} if each component of $\boundary Q - \boundary \ob{Q}$ is a curve on $F$ parallel to $a$ or to some $a^*$. We denote the number of components of $\boundary Q - \boundary \ob{Q}$ parallel to $a$ by $q = q(\ob{Q})$ and the number parallel to $a^*$ by $q^* = q^*(\ob{Q})$. Let $\wihat{q} = q + q^*$. 

The next theorem is a particular case of \cite[Theorem 5.1]{T2}, however for completeness and clarity we sketch the proof.

\begin{theorem}\label{Thm: Constructing Q}
Suppose that $N$ is a 3--manifold with a genus two boundary component $F$ containing essential simple closed curves $a$ and $b$ that cannot be isotoped to be disjoint. Let $\ob{R} \subset N[a]$ be a suitably embedded essential surface and suppose that $R = \ob{R} \cap N$ has a compressing disc or $b$--boundary compressing disc $D$ and that $\wihat{q}(\ob{R}) > 0$. Then there exists an essential surface $\ob{Q} \subset N[a]$ such that the following hold:
\begin{enumerate}
\item If $D$ is a $b$--boundary compressing disc, $Q$ is obtained from $R$ by a $b$-boundary compression and possibly capping off inessential boundary components.
\item If $D$ is a compressing disc and if $N[a]$ is reducible, then $\ob{Q}$ is obtained by a trivial 2--surgery of $\ob{R}$ and the discarding of a sphere component. If $D$ is a compressing disc and if $N[a]$ is irreducible, then $\ob{Q}$ is obtained by a proper isotopy of $\ob{R}$.
\item $-\chi(\ob{Q}) \leq -\chi(\ob{Q})$
\item The sum of the genera of the components of $\ob{Q}$ is no more than the sum of the genera of the components of $\ob{R}$.
\item $\wihat{q}(\ob{Q}) < \wihat{q}(\ob{R})$.
\item If $D$ is a $b$--boundary compressing disc, unless $\boundary D \cap N$ is an arc joining a component of $\boundary R$ to itself, $\ob{Q}$ is properly isotopic to $\ob{R}$ in $N[a]$.
\end{enumerate}
\end{theorem}
\begin{proof}
If $D$ is a compressing disc for $R$ then, since $\ob{R}$ is essential, there is a disc $D' \subset \ob{R}$ with $\boundary D' = \boundary D$. If $N[a]$ is irreducible, $D \cup D'$ bounds a 3--ball in $N[a]$ and isotoping $\ob{R}$ across the 3--ball reduces $\wihat{q}$. If $N[a]$ is reducible, use $D$ to perform a trivial 2--surgery on $\ob{Q}$; discard the 2--sphere component of the resulting surface that contains $D'$. 

Let $D$ be a $b$-boundary compressing disc for $R$ and let $\delta = \boundary D \cap \boundary N$. Since $\wihat{q}(\ob{R}) > 0$, if $\delta$ joins two components of $\boundary \ob{R}$, then $D$ is a boundary compressing disc for $\ob{R}$ in $N[a]$. In this case, since $\delta$ would lie on $\boundary_0 N[a]$, $\ob{R}$ would be either compressible or a boundary-parallel annulus. Either possibility contradicts the assumption that $\ob{R}$ is essential in $N[a]$.

If $\delta$ joins a component $c$ of $\boundary \ob{R}$ to a component $d$ of $\boundary R$ that is parallel to $a^*$ or $a$, there is a proper isotopy of $\ob{R}$ across $D$ that moves $c$ across a copy of $d$ in $\boundary_0 N[a]$. This reduces $\wihat{q}$.

Assume that $\delta$ joins distinct components of $\boundary R - \boundary \ob{R}$. If $\delta$ joins components both parallel to $a$ and lies in an annulus component of $F - \boundary R$, then there is an isotopy of $\alpha$ across $D$ through $\ob{R}$ reducing $\wihat{q}$. Reversing the isotopy gives an isotopy of $\ob{R}$ reducing $\wihat{q}$. Similarly, if $\delta$ joins two components parallel to $a^*$, there is an isotopy of $\wihat{R}$ reducing $\wihat{q}$. If $\delta$ joins two components of $\boundary R$, both parallel to $a$, but does not lie in an annulus component of $F - \boundary R$, then isotoping $\ob{Q}$ across $D$ removes those two components of $\boundary R$ and introduces a third parallel to some $a^*$. If $\delta$ joins a component parallel to $a^*$ to a component parallel to $a$, then there is an isotopy of $\ob{R}$ that removes those two components of $\boundary R$ and introduces a third parallel to either $a$ or $a^*$. Thus, $\wihat{q}$ is reduced by an isotopy of $\ob{R}$ in all cases where $\delta$ joins distinct components of $\boundary R - \boundary \ob{R}$.

Assume that $\delta$ joins a component $c$ of $\boundary R - \boundary \ob{R}$ to itself. Let $C \subset \eta(\alpha)$ be the disc bounded by $c$. Boundary compress $\ob{R} - \inter{C}$ using $D$. Let $\ob{Q}$ be the resulting surface. Clearly, conclusions (1) - (4) hold. We need only show that $\ob{Q}$ is essential. Any compressing disc for $\ob{Q}$ could be isotoped to be disjoint from $\eta(D)$ and so would be a compressing disc for $\ob{R}$, contrary to hypothesis. If $\ob{R} - \inter{C}$ can be obtained from $\ob{Q}$ by tubing $\ob{Q}$ to itself using an arc $\zeta$ lying on $\boundary N[a]$ with endpoints on $\boundary \ob{Q}$. If $\ob{Q}$ is boundary-parallel with region of parallelism $P$, then if $\zeta$ is outside $P$, $\ob{R}$ would also be boundary parallel and if $\zeta$ is interior to $P$, $\ob{R}$ would be compressible. Both contradict the hypothesis that $\ob{R}$ is essential in $N[a]$.
\end{proof}

Suppose that $\ob{R}$ and $\ob{Q}$ are suitably embedded essential surfaces in $N[a]$ and that $\ob{Q}$ is a non-empty surface obtained from $\ob{R}$ by applying Theorem \ref{Thm: Constructing Q} to $\ob{R}$ and possibly discarding one or more components. Then we say that $\ob{Q} < \ob{R}$. If for a given $\ob{R}$, either $\wihat{q}(\ob{R}) = 0$ or there does not exist an essential surface in $N[a]$ such that $\ob{Q} < \ob{R}$, then we say that $\ob{R}$ is \defn{minimal} with respect to $b$-boundary compressions. Extend the relation $<$ so that it is transitive.

The next lemma records some easy facts.
\begin{lemma}
The following are true:
\begin{enumerate}
\item If $\ob{R} \subset N[a]$ is an essential properly embedded surface, then there exists an essential suitably embedded surface $\ob{Q} \subset N[a]$, minimal with respect to $b$-boundary compressions, such that $\ob{Q} \leq \ob{R}$.
\item Suppose that $\ob{Q} \subset N[a]$ is minimal with respect to $b$-boundary compressions but is not $b$-boundary incompressible. Then $\wihat{q}(\ob{Q}) = 0$ and either there is a proper isotopy of $\ob{Q}$ in $N[a]$ that reduces $|\boundary \ob{Q} \cap b|$ or every $b$-boundary compressing disc for $\ob{Q}$ has boundary intersecting $a$.
\end{enumerate}
\end{lemma}
\begin{proof}
To prove (1), simply notice that each application of Theorem \ref{Thm: Constructing Q} reduces $\wihat{q}$ and that the theorem cannot be applied if $\wihat{q} = 0$.

To prove (2), suppose that $\ob{Q} \subset N[a]$ is minimal with respect to $b$-boundary compressions, but has a $b$-boundary compression. Since we cannot apply Theorem \ref{Thm: Constructing Q} we must have $\wihat{q}(\ob{Q}) = 0$. If there were a $b$-boundary compressing disc disjoint from $a$, then $\ob{Q}$ would be $b$-boundary compressible in $N[a]$. A $b$-boundary compressing disc for $\ob{Q}$ adjacent to $\boundary_0 N[a]$ and disjoint from $a$ either gives a method for reducing $|\boundary \ob{Q} \cap b|$ or shows that $\ob{Q}$ is inessential. The latter is impossible.
\end{proof}

It will occasionally be useful to double $N$ along boundary components other than $F$. We make the following definitions. Let $D_aN$ denote the result of gluing a copy of $N_2 = N[a]$ to $N_1 = N$ using the identity map on $\boundary_1 N_1$. Let $F_1$, $a_1$, and $b_1$ denote the copies of $F$, $a$, and $b$ in $N_2$.  Suppose that $\ob{Q} \subset N[a]$ is a suitably embedded surface. Let $\ob{Q}'$ denote the result of isotoping $\ob{Q}$ in $N[a]$ to be properly embedded. Let $\ob{Q}_1$ be the copy of $\ob{Q}$ in $N_1$ and let $\ob{Q}'_2$ be the copy of $\ob{Q}'$ in $N_2$. Let $D_a\ob{Q} = \ob{Q}_1 \cup \ob{Q}'_2 \subset D_a N$. Notice that $D_a\ob{Q}$ is a suitably embedded surface in $D_aN[a_1]$. Let $D_a Q = D_a \ob{Q} \cap D_a N$.

\begin{lemma}\label{Lem: Doubled Param}
Suppose that $N[a]$ is irreducible and boundary-irreducible. Suppose that $\ob{Q} \subset N[a]$ is a suitably embedded surface such that $Q$ has no compressing or $b$-boundary compressing disc in $N$. Assume also that there is no boundary compressing disc for $Q$ adjacent to $\boundary_1 N$. Then $D_a N$ is irreducible and there is no compressing or $b_1$-boundary compressing disc for $D_a Q$ in $D_a N$.
\end{lemma}
\begin{proof}
Since $N[a]$ and $N$ are irreducible and since $\boundary N[a]$ is incompressible, any reducing sphere for $D_a N$ could be isotoped to lie in $N$ or $N[a]$, a contradiction.

Suppose that $E$ is a $b$-boundary compressing disc for $D_a Q$, chosen so that out of all such, $|E \cap \boundary_1 N|$ is minimal. Since $Q$ is $b$-boundary incompressible, $E \cap \boundary_1 N$ is non-empty. An innermost disc argument, using the incompressibility of $\boundary_1 N$, shows that $E \cap \boundary_1 N$ consists of arcs. There is an arc bounding a disc $E' \subset E$ with interior disjoint from $\boundary_1 N$ and with boundary disjoint from $F$. Then $E'$ is a boundary compressing disc for $\ob{Q}$ in $D_a N$. 
\end{proof}

\section{Adding sutures to $\boundary N$}

To use Theorem \ref{Thm: Main Theorem} to prove statements that don't use the language of sutured manifolds requires methods of placing a sutured manifold structure on 3-manifolds with boundary. We explain how to do that in this section. Let $F \subset \boundary N$ be a genus two component. Let $b \subset F$ be an essential simple closed curve. We begin by discussing sutures $\wihat{\gamma} \cup b$ on $F$.

If $b$ is separating, define $\wihat{\gamma} = \nil$. If $b$ is non-separating, define $\wihat{\gamma}$ to be any pair of essential disjoint simple closed curves on $\boundary_0 N[b]$ that are disjoint from and separate the components of $\boundary \eta(b)$. Thus, $F - (\wihat{\gamma} \cup b)$ is either the union of two once punctured tori or the union of two thrice-punctured spheres.

\begin{lemma}\label{Lem: Torus Sutures}
Suppose that $\boundary N - F$ is empty or consists of tori. Let $\wihat{\gamma}$ and $b$ be simple closed curves on $F$ as above. If $N$ is irreducible and if $F - (\wihat{\gamma} \cup b)$ is incompressible in $N$ then $(N,\wihat{\gamma} \cup b)$ is a taut sutured manifold.
\end{lemma}
\begin{proof}
Since $N$ is irreducible, any torus component of $\boundary N$ is incompressible. Since $R_\pm \cap F$ is either a thrice-punctured sphere or a once-punctured torus, $x(R_\pm) = 1$. If $S$ is a norm-minimizing, properly embedded surface in $N$ with $\boundary S = \boundary R_\pm$ and homologous to $R_\pm$, then $x(S) \leq 1$. Without loss of generality, we may assume that $S$ has no 2--sphere or inessential disc components. If $x(S) = 0$, then each component of $S$ is either a disc or an annulus. Since $|\boundary R_\pm|$ is odd, at least one component of $S$ is a disc, contradicting the assumption that $F - (\wihat{\gamma} \cup b)$ is incompressible. Thus, $R_\pm$ and $N$ are taut.
\end{proof}

If $b$ is non-separating, we will often want to be more precise about which curves $\wihat{\gamma}$ we choose. Usually our choice will be determined by particular curves in $F$ as follows.

Let $T = \boundary_0 N[b] - \boundary \beta$; note that $T$ is a twice-punctured torus. Let $\zeta \subset F$ be the union of one or more simple closed curves on $F$ such that no component of $\zeta \cap T$ is an inessential arc or circle.  If an arc component of $\zeta \cap T$ has both endpoints at the same endpoint of $\beta$, then there are such arcs based at both endpoints of $\beta$ and the essential circles formed by any two such arc components are isotopic on $\boundary_0 N[b]$. If such arcs exist, choose $\wihat{\gamma}$ so that each component is isotopic to these essential circles. If no such arc component of $\zeta \cap T$ exists, choose the components of $\wihat{\gamma}$ so that each arc component of $\zeta \cap T$ intersects $\wihat{\gamma}$ exactly once. (To see that this is possible, consider a homeomorphism of $T$ that takes the arcs of $\zeta \cap T$ to ``standard arcs'' joining the punctures, choose curves intersecting the standard arcs exactly once each, and let $\wihat{\gamma}$ be the inverse image of those curves.)  Subject to these constraints, isotope $\wihat{\gamma}$ so that it intersects $\zeta$ minimally. If $\wihat{\gamma} = \nil$ or if $\wihat{\gamma}$ is chosen as described, we say that $\wihat{\gamma}$ \defn{respects} $\zeta$. If $a \subset F$ is an essential simple closed curve intersecting $b$ minimally and not disjoint from $b$, and if $\ob{Q} \subset N[a]$ is a suitably embedded surface with $Q = \ob{Q} \cap N$, then $\wihat{\gamma}$ \defn{respects} $\ob{Q}$ if it respects $\boundary Q \cup a$. See Figure \ref{Fig: RespectingSutures}.

\begin{figure}
\labellist \small\hair 2pt 
\pinlabel {$b$} at 329 35
\pinlabel {$b$} at 390 141
\pinlabel {$\zeta$} [b] at 178 269
\pinlabel {$\wihat{\gamma}$} [b] at 59 277
\pinlabel {$\wihat{\gamma}$} [t] at 259 150
\endlabellist 
\centering 
\includegraphics[scale=.6]{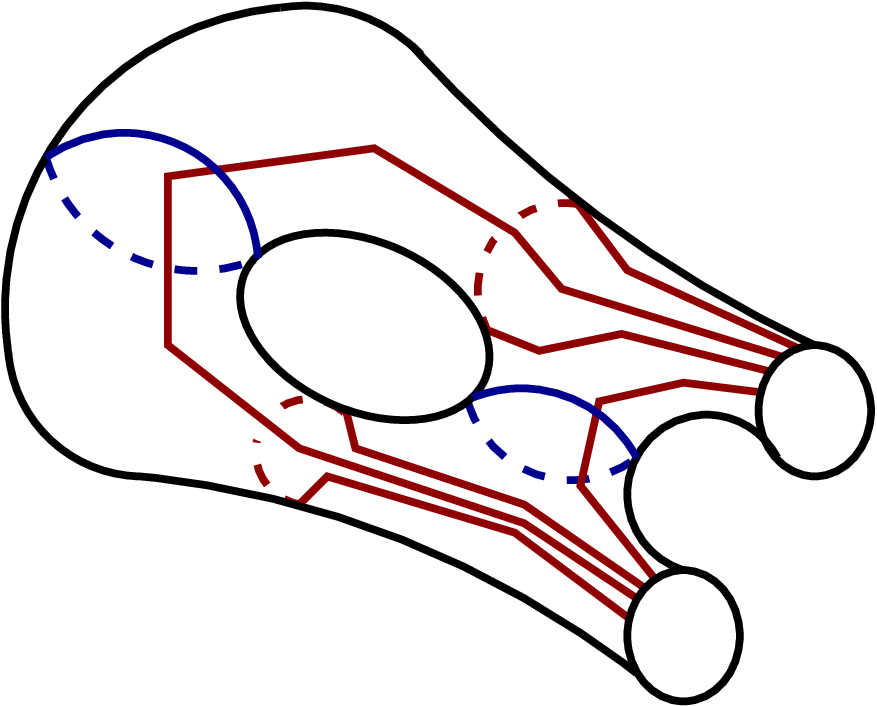}
\caption{An example of sutures $\wihat{\gamma} \subset F$ respecting a 1--manifold $\zeta$ on the twice punctured torus $F - \inter{\eta}(b)$.}
\label{Fig: RespectingSutures}
\end{figure}
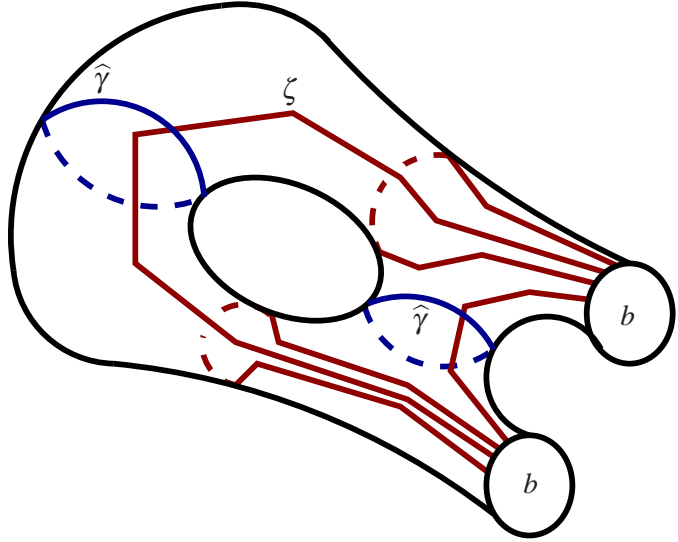

Let $\mc{M}_b(\zeta)$ denote the number of arcs of $\zeta \cap T$ that have both endpoints on the same component of $\boundary \beta$. 

\begin{lemma}\label{Lem: Differing Intersections}
Suppose that $\zeta$ is an embedded 1--manifold on $F$ intersecting $b$ minimally and let $\zeta_0$ be a component of $\zeta$. Then for each component $\zeta_0$ of $\zeta$ we have the following:
\begin{enumerate}
\item If $b$ and $\zeta_0$ are given orientations then $\mc{M}_b(\zeta_0)$ is the number of changes of sign of $b \cap \zeta_0$ as $\zeta_0$ is traversed.
\item $\mc{M}_b(\zeta)$ is even
\item If $\mc{M}_b(\zeta_0) = 0$ and $|b \cap \zeta_0| > 0$, then both $\zeta_0$ and $b$ are non-separating and all points of intersection have the same sign.
\item If  $\wihat{\gamma}$ respects $\zeta$ then 
\[
\mc{M}_b(\zeta_0) = |\zeta_0 \cap b| - |\zeta_0 \cap \wihat{\gamma}|.
\]
\end{enumerate}
\end{lemma}

\begin{proof}
Let $b_1$ and $b_2$ be the two endpoints of $\beta$ on $\boundary_0 N[b]$. Let $T = \boundary_0 N[b] - \boundary \beta$. If an arc of $\zeta_0 \cap T$ has both endpoints on either $b_1$ or $b_2$, then the corresponding points of intersection of $\zeta_0 \cap b$ have opposite sign and are adjacent on $\zeta_0$. Conversely, any change of sign on $\zeta_0$ must correspond to an arc of $\zeta_0 \cap T$ with both endpoints at either $b_1$ or $b_2$, since each component of $T$ is a once or twice-punctured torus. Since $\zeta_0$ is a circle, there must be an even number of changes of sign. Hence, (1). This is true for each component of $\zeta$, so $\mc{M}_b(\zeta)$ is even. Hence, (2). 

If $\mc{M}_b(\zeta_0) = 0$ and $|b \cap \zeta_0| > 0$ then $b$ and $\zeta_0$ intersect and all points of intersection have the same sign by (1). Consequently, both must be non-separating. Hence, (3).

If $b$ is separating, (4) follows immediately, so assume $b$ is non-separating. There are the same number of endpoints of $\zeta_0 \cap T$ at $b_1$ as at $b_2$. Each arc either has both endpoints at one of $b_1$ or $b_2$ or has one endpoint at $b_1$ and the other $b_2$. Thus, there are the same number of arcs of $\zeta_0 \cap T$ with both endpoints at $b_1$ as there are arcs of $\zeta_0 \cap T$ with both endpoints at $b_2$. This is true for each component of $\zeta$. If $\mc{M}_b(\zeta) > 0$, there exists at least one arc $\zeta'$ of $\zeta \cap T$ with both endpoints at $b_1$ and at least one arc $\zeta''$ of $\zeta \cap T$ with both endpoints at $b_2$ and all other arcs of $\zeta \cap T$ are disjoint from these arcs. Since, $\zeta$ intersects $b$ minimally, in $\boundary_0 N[b]$, $\zeta'$ and $\zeta''$ are disjoint essential loops and are, therefore, parallel. Assume that $\wihat{\gamma}$ respects $\zeta$. The curves $\wihat{\gamma}$ are parallel to these loops and intersect the other arcs of $\zeta \cap T$ minimally. Since $\boundary_0 N[b] - (\zeta' \cup \zeta'')$ is a pair of annuli and since there is one component of $\wihat{\gamma}$ in each annulus, we have
\[
\mc{M}_b(\zeta_0) = |\zeta_0 \cap b| - |\zeta_0 \cap \wihat{\gamma}|,
\]
if $\mc{M}_b(\zeta) > 0$.

If $\mc{M}_b(\zeta) = 0$, then by the construction of $\zeta$,
\[
\mc{M}_b(\zeta_0) = |\zeta_0 \cap b| - |\zeta_0 \cap \wihat{\gamma}| = 0.
\]
\end{proof}

The next lemma shows how to usefully define sutures $\tild{\gamma}$ on non-torus components of $\boundary N - F$. Let $\gamma = \wihat{\gamma} \cup \tild{\gamma} \cup b$.

\begin{lemma}[{\cite[Lemma 4.1]{T1}}]\label{Lem: Choosing Sutures}
Assume\footnote{This hypothesis was left out of the statement of the lemma in \cite{T1}. It is used in the proof when \cite[Lemma 5.6]{S4} is applied.} that $\boundary N - (\wihat{\gamma} \cup b)$ is incompressible in $N$. Suppose also that $N$ is irreducible and that if $\boundary N - F$ contains a non-torus component then there is no essential annulus in $N$ with boundary on $\wihat{\gamma} \cup b$. Then $\tild{\gamma}$ can be chosen so that $(N,\gamma)$ is taut.  Furthermore, if $c \subset \boundary N - F$ is a collection of disjoint, non-parallel curves such that:
\begin{itemize}
\item $|c| \leq 2$
\item All components of $c$ are on the same component of $\boundary N - F$
\item No component of $c$ cobounds an essential annulus in $N$ with a curve of $\wihat{\gamma} \cup b$
\item If $|c| = 2$ then there is no essential annulus in $N$ with boundary $c$
\item If $|c| = 2$ and $b$ is separating, there is no essential thrice-punctured sphere in $N$ with boundary $c \cup b$.
\end{itemize}
then $\tild{\gamma}$ can be chosen to be disjoint from $c$.
\end{lemma}

If $\wihat{\gamma}$ respects a suitably embedded surface $\ob{Q} \subset N[a]$, then we also say that $\gamma$ \defn{respects} $\ob{Q}$.

The following lemma records a handy bit of accounting that we will use repeatedly.

\begin{lemma}\label{Lem: Accounting}
Suppose that $\gamma$ respects a suitably embedded surface $\ob{Q} \subset N[a]$ and that $\mc{M}_b(a) > 0$. Then the inequality
\[
-2\chi(Q) + |\boundary Q \cap \gamma| \geq 2|\boundary Q \cap b|
\]
implies the inequality:
\[
-2\chi(\ob{Q}) + |\boundary \ob{Q} \cap \tild{\gamma}| \geq q(\mc{M}_b(a) - 2) + q^*(\mc{M}_b(a^*) - 2) + \mc{M}_b(\boundary \ob{Q}).
\]
\end{lemma}
\begin{proof}
This is easily proven by piecing together Lemma \ref{Lem: Differing Intersections} and the following two facts:
\begin{itemize}
\item $\chi(Q) = \chi(\ob{Q}) - \wihat{q}(\ob{Q})$
\item For any simple closed curve $\zeta$ on $\boundary N$, 
\[|\zeta \cap \gamma| = |\zeta \cap \wihat{\gamma}| + |\zeta \cap \tild{\gamma}| + |\zeta \cap b|.\]
\end{itemize}
\end{proof}

\section{Degenerating Handle Additions}\label{Degenerating}

In this section, we assume that $N$ is simple and that attaching a 2--handle to $N$ along $b$ creates some sort of degeneration (usually either an essential disc or sphere). We use this to, in many instances, deduce an inequality relating the intersection number of $a$ and $b$ to the euler characteristic of an essential surface in $N[a]$.

\begin{proposition}\label{Thm: Comparing 1}
Suppose that $N$ is a compact orientable simple 3-manifold with $F \subset \boundary N$ a genus 2 component. Let $a,b$ be two essential simple closed curves on $\boundary N$ that cannot be isotoped to be disjoint.

Let $\ob{Q} \subset N[a]$ be a suitably embedded essential surface and assume that $Q = \ob{Q} \cap N$ is incompressible and $b$-boundary incompressible.  Assume also that $\boundary \ob{Q} \cap \boundary_1 N$, if non-empty, consists of curves all parallel to a simple closed curve $c_1 \subset \boundary_1 N$. 

Assume that one of the following holds:
\begin{enumerate}
\item[(a)] $N[b]$ is reducible
\item[(b)] $N[b]$ is a solid torus
\item[(c)] A simple closed curve $c_b$ disjoint from $c_1$ and on the same component of $\boundary_1 N$ compresses in $N[b]$. 
\item[(d)] There is a non-zero class $y \in H_2(N[b],\boundary N[b])$ such that no norm-minimizing conditioned surface representing $\pm y$ is disjoint from $\beta$. If $b$ is separating, we also require that the projection of $\boundary y$ to the first homology of each component of $\boundary_0 N[b]$ is nontrivial.
\end{enumerate}

Then, one of the following occurs:
\begin{enumerate}
\item $a$ and $b$ are both non-separating and all intersection points of $a \cap b$ have the same sign.
\item $\ob{Q} \subset N$ and $\ob{Q}$ is disjoint from $b$.
\item 
\[
-2\chi(\ob{Q}) \geq q(\mc{M}_b(a) - 2) + q^*(\mc{M}_b(a^*) - 2) + \mc{M}_b(\boundary \ob{Q}). 
\]
\end{enumerate}
\end{proposition}
\begin{proof}
Since $N$ is simple, notice that hypothesis (b) implies hypothesis (d). If hypothesis (a), (b), or (d) is satisfied, let $c_2 = \nil$. If none of those are satisfied, but the curve $c_1$ compresses in $N[b]$, let $c_2 = c_1$. If none of (a), (b), or (d) is satisfied and if $c_1$ does not compress in $N[b]$, let $c_2 = c_b$. Define $c = c_1 \cup c_2$ and notice that $c$ has two components if and only if $c_1$ does not compress in $N[b]$ and $c_1$ and $c_2$ are non-empty, disjoint, non-parallel loops on the same component of $\boundary_1 N$. 

We claim that $c$ satisfies the hypotheses of Lemma \ref{Lem: Choosing Sutures}. Since $N$ contains no essential disc or annulus, we need only check that if $|c| = 2$ and $b$ is separating then there is no essential thrice-punctured sphere in $N$ with boundary $c \cup b$. Since $c_b$ compresses in $N[b]$, the existence of a thrice-punctured sphere in $N$ with boundary on $c \cup b$ would imply that $c_1$ compresses in $N[b]$. But, by our choice of $c_2$, this contradicts the hypothesis that $|c| = 2$. Thus, the hypotheses of Lemma \ref{Lem: Choosing Sutures} are satisfied. Let $\gamma = \tild{\gamma} \cup \wihat{\gamma} \cup b$ be sutures respecting $\ob{Q}$ and disjoint from $c$. The sutured manifold $(N,\gamma)$ is taut. 

If $|\boundary Q \cap b| = 0$, then $\wihat{q}(\ob{Q}) = 0$ and $|\boundary \ob{Q} \cap b| = 0$. This implies conclusion (2). Assume that neither conclusion (1) nor conclusion (2) holds. Since $\gamma$ respects $Q$, and since $N$ is simple, no component of $Q$ is a sphere or disc disjoint from $\gamma$. Thus, the hypotheses of Theorem \ref{Thm: Main Theorem} are satisfied. By hypothesis, $Q$ has no compressing or $b$-boundary compressing disc. Since $N$ is simple, $\beta$ cannot intersect a reducing sphere in $N[b]$ exactly twice. By hypothesis and the construction of $\tild{\gamma}$ to be disjoint from $c_b$, either $(N[b],\gamma - b)$ is not $\nil$-taut or no first surface in a taut sutured manifold hierarchy for $(N[b],\gamma -b )$ representing $y$ is disjoint from $\beta$. Thus, by Theorem \ref{Thm: Main Theorem},
\[
-2\chi(Q) + |Q \cap \gamma| \geq 2|Q \cap b|.
\]

By Lemma \ref{Lem: Accounting}, this implies that 
\[
-2\chi(\ob{Q}) \geq q(\mc{M}_b(a) - 2) + q^*(\mc{M}_b(a^*) - 2) + \mc{M}_b(\boundary \ob{Q}). 
\]
\end{proof}

\begin{corollary}\label{Cor: Comparing 1}
Suppose that $N$ is a compact, orientable simple 3-manifold with a genus two boundary component $F$. Let $a,b \subset F$ be two essential curves that cannot be isotoped to be disjoint. Assume that $N[a]$ contains an essential sphere or disc $\ob{R}$.

Assume also that one of the following holds: 
\begin{itemize}
\item $N[b]$ is reducible
\item $N[b]$ is a solid torus
\item A simple closed curve $c_b$ disjoint from, and on the same component of $\boundary_1 N[b]$ as $\boundary \ob{R}$, compresses in $N[b]$. 
\end{itemize}

Then $a$ and $b$ are both non-separating, and all points of intersection between them have the same sign.
\end{corollary}

\begin{proof}
Assume that not all points of intersection between $a$ and $b$ have the same sign. This implies that $\mc{M}_b(a) \geq 2$.

If $N[a]$ is reducible, let $\ob{R}$ be an essential sphere in $N[a]$ and if $N[a]$ is irreducible, let $\ob{R}$ be an essential disc in $N[a]$. Choose $\ob{Q} \leq \ob{R}$ to be minimal with respect to $b$-boundary compressions. Notice that if $\ob{Q}$ is a disc, then $\ob{R}$ is as well and if $\boundary \ob{Q} \subset \boundary_1 N[b]$ then $\boundary \ob{Q} = \boundary \ob{R}$.

Since $N$ is irreducible and boundary-irreducible, $\wihat{q}(\ob{Q}) \neq 0$. If $a^*$ is non-empty and is disjoint from $b$, then $a$ and $b$ lie in a common once-punctured torus. In which case, $a$ and $b$ are both non-separating and intersect always with the same sign. This contradicts our assumption. Hence, if $q^* \neq 0$, then $\mc{M}_b(a^*) \geq 2$ since $a^*$ is separating.
 
If $q > 0$, by Theorem \ref{Thm: Comparing 1}, we have:
\[
-2\chi(\ob{Q}) \geq q(\mc{M}_b(a) - 2) \geq 0.
\]
But $\ob{Q}$ is a sphere or a disc, so this is a contradiction.

If $q^* > 0$, then by Theorem \ref{Thm: Comparing 1}, we have
\[
-2\chi(\ob{Q}) \geq q^*(\mc{M}_b(a^*) - 2) \geq 0.
\]
This is also a contradiction. Thus, all points of intersection between $a$ and $b$ have the same sign.
\end{proof}

Here is a version of Proposition \ref{Thm: Comparing 1} allowing for $\boundary \ob{Q}$ to have non-isotopic components on $\boundary_1 N$.

\begin{proposition}\label{Thm: Comparing 2}
Suppose that $N$ is a compact, orientable, simple 3-manifold and that $F \subset \boundary N$ has genus 2. Let $a,b \subset F$ be two essential curves that cannot be isotoped to be disjoint. Suppose that $\ob{Q} \subset N[a]$ is a suitably embedded essential surface and that $Q = \ob{Q} \cap N$ is incompressible, $b$-boundary-incompressible, and has no boundary compressing disc adjacent to $\boundary_1 N$. If $N[b]$ is reducible, then one of the following holds:
\begin{enumerate}
\item $a$ and $b$ are non-separating and all points of $a \cap b$ have the same sign.
\item $\ob{Q} \subset N$ and $\boundary \ob{Q}$ is disjoint from $b$.
\item
\[
-4\chi(\ob{Q}) \geq q(\mc{M}_b(a) - 2) + q^*(\mc{M}_b(a^*) - 2) + \mc{M}_b(\boundary \ob{Q}). 
\]
\end{enumerate}
\end{proposition}
\begin{proof}
Assume neither (1) nor (2) holds. By Corollary \ref{Cor: Comparing 1}, we may assume that $N[a]$ is irreducible and boundary-irreducible. Form the doubled manifold $D_a N$. Notice that $\boundary D_a N$ consists of a single genus two surface and some tori. By Lemma \ref{Lem: Doubled Param}, $D_a Q$ has no compressing or $b_1$-boundary compressing disc. Choose $\wihat{\gamma} \subset F$ to respect $D_a Q$. By Lemma \ref{Lem: Torus Sutures} $(D_a N, \wihat{\gamma} \cup b_1)$ is taut. Apply Theorem \ref{Thm: Main Theorem} to $D_a N$. One of the following must hold:
\begin{enumerate}
\item[(i)] $(D_a N[b_1], \beta)$ has a (lens space, core) connected summand

\item[(ii)]
\[
-2\chi(D_a \ob{Q}) \geq q(\mc{M}_b(a) - 2) + q^*(\mc{M}_b(a^*) - 2) + \mc{M}_b(\boundary \ob{Q}). 
\]
\end{enumerate}

The second conclusion together with the fact that $\chi(D_a \ob{Q}) = 2\chi(\ob{Q})$ gives us our result immediately.

Assume, therefore, that there is a reducing sphere $P$ for $D_a N[b]$ intersected twice by $\beta$. Out of all such spheres choose $P$ to intersect $\boundary_1 N$ minimally. Since $N$ is simple, and $\beta$ is disjoint from $\boundary_1 N$, and $N[a]$ is boundary-irreducible, an innermost disc argument shows that any circle of $P \cap \boundary_1 N$ innermost on $P$ bounds a disc intersecting $\beta$. Hence, $P \cap \boundary_1 N$ has exactly two circles innermost on $P$. Let $D$ be one of them. The disc $D$ is punctured exactly once by $\beta$. Thus $D \cap N_1$ is an essential annulus in $N$, contrary to the hypothesis that $N$ is simple.
\end{proof}

We can apply this to prove a special case of a conjecture of Scharlemann and Wu (as described in the introduction).
\begin{corollary}\label{Cor-ScharlWu}
Suppose that $N$ is a compact orientable simple 3-manifold and that $F \subset \boundary N$ is a genus two boundary component. Let $a,b \subset F$ be essential simple closed curves that cannot be isotoped to be disjoint. Assume that $b$ is separating and that $N[b]$ is reducible or a solid torus. Then $N[a]$ is irreducible and boundary-irreducible and if $a$ is a degenerating curve, then $a$ is non-separating and $|a \cap b| = 2$. Furthermore, if $N[a]$ contains an essential annulus, then it contains one with boundary disjoint from $b \cap \boundary N[a]$.
\end{corollary}
\begin{proof}
By Corollary \ref{Cor: Comparing 1}, $N[a]$ is irreducible and boundary-irreducible.

Let $\ob{R}$ be an annulus or torus in $N[a]$. Let $\ob{Q} \leq \ob{R}$ be minimal. If $\ob{Q}$ had a boundary-compressing disc adjacent to $\boundary_1 N$, it must be an annulus. Boundary-compressing it creates an essential disc, contrary to our initial observation. Thus, we may assume that $\ob{Q}$ does not have an boundary compressing disc adjacent to $\boundary_1 N[b]$.

Since $b$ is separating and since $N$ is simple, by Propositions \ref{Thm: Comparing 1} and \ref{Thm: Comparing 2}, we have
\[
0 = -\epsilon\chi(\ob{Q}) \geq q(|a \cap b| - 2) + q^*(|a^* \cap b| - 2) + |\boundary \ob{Q} \cap b| \geq 0, 
\]
for some $\epsilon \in \{2,4\}$. This implies that $|\boundary \ob{Q} \cap b| = 0$.

Recall that any two minimally intersecting, non-disjoint simple closed separating curves on a genus two surface must intersect at least four times. Since $a^*$ is empty or separating, we must have $q^* = 0$. Since $N$ is simple, we must have $q > 0$. Consequently, $|a \cap b| = 2$. Since $b$ is separating, this implies that $a$ is non-separating.
\end{proof}

\section{Refilling Meridians}\label{Refilling}

We recall (from the introduction) the notion of ``refilling meridians''. Let $M$ be a compact orientable 3--manifold, and let $W \subset M$ be an embedded genus 2 handlebody. Let $A$ and $B$ be essential discs in $W$ that cannot be isotoped to be disjoint. Assume that $A$ and $B$ intersect minimally. Define $a = \boundary A$, $b = \boundary B$, $F = \boundary W$, and $N = M - \inter{W}$. We say that the 3--manifolds $N[a]$ and $N[b]$ are obtained by \defn{refilling} the meridians $A$ and $B$ respectively of $W$ in $M$. Notice that an outermost arc of intersection between $A$ and $B$ on $A$ cuts off a subdisc of $A$ that is a meridian of one of the solid tori obtained by boundary-reducing $W$ using $B$. In particular, this means that $\mc{M}_b(a) \geq 2$. Recall that $L_a$ and $L_b$ denote the core or cores of the solid torus or tori $W - \inter{\eta}(A)$ and $W - \inter{\eta}(B)$ respectively. If $a$ is non-separating and if $a^* \subset F$ bounds a once-punctured torus in $F$ containing $a$, let $A^*$ be a disc in $W$ bounded by $a^*$.

In \cite{MSc5}, Scharlemann considered refilling meridians of a genus two handlebody $W$ in a wide variety of compact, orientable 3--manifolds $M$. He showed in a number of situations that if both $N[a]$ and $N[b]$ are reducible or boundary-irreducible (with minor assumptions on the embedding of $W$ in $M$) then either $M = S^3$ and $W$ is unknotted or $\alpha$ and $\beta$ are ``aligned'' in $W$. He conjectured that this is always the case (given his hypotheses on the pair $(M,W)$). In \cite{T1}, under slightly different hypotheses, significant progress was made on the conjecture. One case not covered by \cite{T1}, was the case when both $N[a]$ and $N[b]$ are solid tori. Corollary \ref{Cor: MSC 2}, a corollary to the next proposition, shows that if $|a \cap b| > 0$, it is impossible for both $N[a]$ and $N[b]$ to be solid tori. In fact, Corollary \ref{Cor: MSC 2} gives a stronger result with weaker hypotheses than \cite[Theorem 6.1]{T1}.

\begin{proposition}\label{Thm: MSC 2}
Suppose that $W$ is a genus 2 handlebody embedded in a compact, orientable manifold $M$. Let $A$ and $B$ be essential discs in $W$ which cannot be isotoped to be disjoint. Suppose that $\ob{Q}$ is a suitably embedded essential surface in $N[a]$ such that $Q = \ob{Q} \cap N$ is incompressible and $b$-boundary-incompressible and not disjoint from $b$. Furthermore, assume that all components of $\boundary \ob{Q} \cap \boundary_1 N[b]$ are parallel to a single simple closed curve $c_a$.

Assume the following:
\begin{enumerate}
\item[(H1)] If a curve $c_b \subset \boundary M$ compresses in $N[b]$ then $c_a$ and $c_b$ are on the same component of $\boundary M$.
\item[(H2)] If $P \subset M$ is a sphere such that either $P$ is non-separating, or $P$ bounds a lens space summand of $M$, or $P$ is embedded and essential in $N[b]$ but inessential in $M$, then $|P \cap (L_b \cup \beta)| \geq 3$.
\item[(H3)] If $P$ is an essential disc in $M$ then $|P \cap (L_a \cup \alpha)| \geq 2$ and $|P \cap (L_b \cup \beta)| \geq 2$.
\item[(H4)] $N = M - \inter{W}$ is irreducible.
\end{enumerate}

Suppose one of the following:
\begin{enumerate}
\item $N[b]$ is reducible
\item $N[b]$ is boundary-reducible
\item There is a non-zero class $y \in H_2(N[b],\boundary N[b])$ such that no conditioned norm-minimizing representative is disjoint from $\beta$. If $b$ is separating, also assume that the projection of $\boundary y$ to the first homology of each component of $\boundary_0 N[b]$ is non-zero.
\end{enumerate}

Then either $c_a$ and $c_b$ cannot be isotoped to be disjoint or
\[-2\chi(\ob{Q}) \geq q(\mc{M}_b(a) - 2) + q^*(\mc{M}_b(a^*) - 2) + \mc{M}_b(\boundary \ob{Q}).\]
\end{proposition}

The proof is similar to that of Corollary \ref{Cor: Comparing 1}.
\begin{proof}
Choose sutures $\wihat{\gamma}$ on $\boundary_0 N$ so that $\wihat{\gamma} \cup b$ respects $Q$. Since $\ob{Q}$ is not disjoint from $\alpha$, any outermost disc of $A - B$ or $A^* - B$ is a meridian of $\eta(L_b)$. Thus, if $\wihat{\gamma} \neq \nil$, both components are meridians of $L_b$. If a component of $F - \wihat{\gamma} \cup b$ were compressible, we would violate (H2), so $F - \wihat{\gamma} \cup b$ is incompressible in $N$.

Let $c_1 = c_a$. If $c_1$ compresses in $N[b]$, or if $N[b]$ is reducible, let $c_2 = c_1$. Otherwise, let $c_2 = c_b$. Thus, $|c| = 2$ if and only if $c_1$ are non-empty and $c_2$ compresses in $N[b]$ but $c_1$ does not.

Hypothesis (H3) implies that $\boundary_1 N$ is incompressible in $N$. If there were an essential annulus in $N$ with boundary on $\wihat{\gamma} \cup b$, it could be capped off in $M$ to be a non-separating sphere intersecting $L_b$ twice. This contradicts Hypothesis (H2).

If a curve of $c$ bounded an essential annulus with a curve of $\wihat{\gamma} \cup b$, it could be capped off in $M$ to be an essential disc intersecting $L_b \cup \beta$ exactly once. This contradicts Hypothesis (H3).

If $|c| = 2$ and there were an essential annulus in $N$ with boundary $c$, then $c_1$ would compress in $N[b]$, contradicting our construction of $c$.

If $|c| = 2$ and there were a thrice-punctured sphere with boundary $c \cup b$, then, once again $c_1$ would compress in $N[b]$. This contradicts our construction of $c$.

We can, therefore, apply Lemma \ref{Lem: Choosing Sutures} to obtain sutures $\gamma \subset \boundary N$ respecting $\ob{Q}$ and disjoint from $c$.  Notice that either $(N[b], \gamma - b)$ is not taut or that there is no conditioned taut representative of $y$ disjoint from $\beta$.

Since $\ob{Q}$ is essential in $N[a]$, since $Q$ is incompressible and $b$-boundary incomrpessible, and since $(N,\gamma)$ is taut, no component of $Q$ is a sphere or disc disjoint from $\gamma$.

Thus, by Theorem \ref{Thm: Main Theorem} and Lemma \ref{Lem: Accounting} one of the following occurs:
\begin{enumerate}
\item[(i)] $(N[b],\beta)$ has a (lens space, core) summand.
\item[(ii)] $(N[b],\gamma - b)$ is taut and there is a conditioned norm-minimizing surface disjoint from $\beta$ and representing $y$.
\item[(iii)] \[-2\chi(\ob{Q}) \geq q(\mc{M}_b(a) - 2) + q^*(\mc{M}_b(a^*) - 2) + \mc{M}_b(\boundary \ob{Q}).\]
\end{enumerate}
The first of these is ruled out by Hypothesis (H2). If $N[b]$ is reducible or if $c_b \neq \nil$, then the second is also ruled out. If $N[b]$ is a solid torus, then an unknotting disc for $L_b$ represents some non-zero class $y \in H_2(N[b],\boundary N[b])$. Such a disc is clearly norm-minimizing. We assume that no such norm-minimizing surface disjoint from $\beta$ exists and so (ii) cannot hold with any of our hypotheses. Consequently (iii) holds, as desired.
\end{proof}

The following corollary is phrased to make the connection to \cite[Conjecture 2]{MSc5} explicit. 

\begin{corollary}\label{Cor: MSC 2}
Assume that $W$ is a genus 2 handlebody embedded in a compact, orientable 3--manifold $M$ such that every $S^2$ in $M$ separates, $M$ contains no lens space connected summands, any pair of curves in $\boundary M$ that compress in $M$ are isotopic in $\boundary M$, and $N = M - \inter{W}$ is irreducible and boundary-irreducible. If $a$ and $b$ are essential closed curves on $\boundary W$ bounding discs in $W$ that cannot be isotoped to be disjoint then one of the following occurs:
\begin{enumerate}
\item One of $N[a]$ and $N[b]$ is irreducible and boundary-irreducible
\item There exists an essential annulus $A \subset N$ such that one component of $\boundary A$ lies on a component of $\boundary M$ and the other lies on $\boundary W$, is disjoint from $a$ or $b$ and bounds a disc in $W$.
\end{enumerate}
\end{corollary}
\begin{proof}
Assume that $N[b]$ is reducible or boundary-reducible. If $N$ is boundary-reducible, but not reducible then either $N[b]$ is a solid torus or there is a unique simple closed curve $c \subset \boundary_1 N$ that compresses in $M$ (and in $N[b]$). This implies (H1).

Since $M$ contains no lens space summands and no non-separating 2--spheres and since $N$ is boundary-irreducible, (H2) is satisfied. Any essential disc in $M$ must have boundary parallel to $c$. Suppose that $P$ is such a disc with either $|P \cap (L_a \cup \alpha)| = 1$ or $|P \cap (L_b \cup \beta)| = 1$. Then (possibly after a small isotopy) the intersection of $P$ with $N$ is an annulus in $N$ with one boundary component on $\boundary W$ disjoint from $a$ or $b$ and bounding a disc in $W$. This is conclusion (2). Hence, we may assume that (H3) is satisfied. (H4) is satisfied by hypothesis.

If $N[a]$ is reducible, let $\ob{R}$ be an essential sphere. If $N[b]$ is irreducible, but boundary-reducible, let $\ob{R}$ be an essential disc. Let $\ob{Q} \leq \ob{R}$ be minimal with respect to $b$-boundary compressions. Since $N$ is irreducible and boundary-irreducible, $\wihat{q}(\ob{Q}) > 0$. Thus, $Q$ is $b$-boundary incompressible and not disjoint from $b$. Since $\ob{Q}$ is either a disc or a sphere, either $\ob{Q}$ is a sphere or $\boundary \ob{Q}$ is parallel to $c$. Thus, by Proposition \ref{Thm: MSC 2}, 
\[
-2 \geq q(\mc{M}_b(a) - 2) + q^*(\mc{M}_b(a^*)  - 2) + \mc{M}_b(\boundary \ob{Q} + 2).
\]
Each term of the right hand side of the inequality is non-negative. We have, therefore, encounted a contradiction.
\end{proof}

If $M = S^3$ we define a boring arc $\beta$ to be \defn{complicated} if one of the following holds:
\begin{itemize}
\item $L_b$ is a split link and $\beta$ does not intersect a splitting sphere exactly once.
\item $L_b$ is not a split link and $\beta$ is not disjoint from any minimal genus Seifert surface for $L_b$.
\end{itemize}
It turns out that the exterior of a knot or link and complicated boring arc is boundary-irreducible. (See also \cite{J}, \cite{S2}, and \cite[Theorem 10.1]{T3}.)
\begin{lemma}\label{Lem: comp arc implies incomp boundary}
Suppose that $L_b \subset S^3$ is a knot or 2-component link and that $\beta$ is a complicated boring arc. Let $N$ be the exterior of $L_b \cup \beta$. Then $N$ is boundary-irreducible.
\end{lemma}
\begin{proof}
Let $W = \eta(L_b \cup \beta)$ and let $F = \boundary W = \boundary N$. Suppose that $D$ is a compressing disc for $F$ in $N$, chosen so that out of all such discs, $|\boundary D \cap b|$ is minimal.

\textbf{Case 1:} $\boundary D \cap b = \nil$.

If $\boundary D$ bounds a disc in $W$, then either $\boundary D$ is parallel to $b$ or $b$ is non-separating and $\boundary D$ bounds a once-punctured torus in $F$ containing $b$. In the former case, $L_b$ is a split link and $\beta$ intersects a splitting sphere exactly once. In the latter case, it is not difficult to see that there is a minimal genus Seifert surface for $L_b$ disjoint from $\beta$. Thus, $\boundary D$ does not bound a disc in $W$. This implies that one component of $L_b$ is an unknot and that $D$ is an unknotting disc for that component disjoint from $\beta$. If $b$ is non-separating this obviously contradicts the definition of complicated boring arc. If $b$ is separating, boundary compressing a component of $\boundary \eta(L_b)$ using $D$ produces a splitting sphere for $L_b$ intersecting $\beta$ exactly once. This also contradicts the definition of complicated boring arc. Consequently, this case cannot occur.

\textbf{Case 2:} $\boundary D \cap b \neq \nil$.

Let $Q = \boundary D$. Notice that by minimality of $|D \cap b|$, $Q$ is $b$-boundary incompressible. It is obviously incompressible. Let $\wihat{\gamma} \subset F$ be sutures disjoint from $b$ and respecting $Q$. Let $\gamma = \wihat{\gamma} \cup b$. It is argued in the proof of \cite[Theorem 10.1]{T3} that $|\boundary Q \cap \wihat{\gamma}| \leq |\boundary Q \cap b|$. Thus, since $Q$ is a disc
\[
-2\chi(Q) + |Q \cap \gamma| \leq 2|Q \cap b|.
\]
Thus, by Theorem \ref{Thm: Main Theorem}, $(N[b],\gamma - b)$ is $\nil$-taut and a proper isotopy of $\beta$ makes it disjoint from a minimal genus Seifert surface for $L_b$. This contradicts the definition of complicated boring arc, and so this case cannot occur either.
\end{proof}

The next corollary shows that many knots obtained by boring using a complicated boring arc are simple (i.e. have simple exteriors).

\begin{corollary}\label{Cor: Non-simple Knot}
Suppose that $L_a$ is a non-simple knot or 2--component link in $S^3$ obtained by boring a knot or 2--component link $L_b \subset S^3$ using a complicated boring arc. Then $L_a$ is not a split link or unknot and one of the following occurs:
\begin{enumerate}
\item The exterior of $L_a$ has an essential torus or annulus disjoint from $\alpha$. 
\item $L_a$ or $L_b$ is a knot and $\mc{M}_b(a) = 2$. If $L_b$ is a 2--component link, this implies $|a \cap b| = 2$.
\item $L_a$ and $L_b$ are both knots, there exists a 2--component link $L_{a^*}$ containing $L_a$ as a component, $L_{a^*}$ is non-simple, and $a^*$ is an essential separating meridional curve on $\boundary \eta(L_a \cup \alpha)$ disjoint from $a$ such that $\mc{M}_b(a^*) = 2$.
\end{enumerate}
Furthermore, if $L_a$ is a knot and $L_b$ is a link, then if there is an essential annulus in the exterior of $L_a$, there is one with meridional boundary.
\end{corollary}

\begin{remark}
It may be worth remarking that the conclusion that $\mc{M}_b(a) = 2$ implies that if $A$ and $B$ are the discs in $\eta(L_a \cup \alpha)$ bounded by $a$ and $b$ then all arcs of intersection of $A\cap B$ are parallel in $A$. This follows from the first conclusion of Lemma \ref{Lem: Differing Intersections}. If $b$ is separating, then also by that lemma we have $\mc{M}_b(a) = |a \cap b|$.
\end{remark}

\begin{proof}
Let $W$ be the genus two handlebody that is a regular neighborhood of $L_b \cup \beta$. Notice that $N = S^3 - \inter{W}$ is irreducible. By Lemma \ref{Lem: comp arc implies incomp boundary}, $\boundary N$ is incompressible. By Proposition \ref{Cor: MSC 2}, $L_a$ is not a split link or unknot. Let $\ob{R}$ be an essential annulus or torus in $N[a]$ and let $\ob{Q} \leq \ob{R}$ be minimal with respect to $b$-boundary compressions. Notice that if $\ob{R}$ was an annulus, then $\ob{Q}$ is as well.

Suppose that $L_a$ is a knot and that $\ob{Q}$ is an annulus. Then, if $\boundary \ob{Q} \cap b = \nil$, $\boundary \ob{Q}$ must be meridional on $L_a$ since it is disjoint from the boundaries of outermost discs of $B - A$ lying in a neighborhood of $L_a$. 

If $Q \cap b = \nil$, then $\ob{Q} \subset N$ and $\boundary \ob{Q} \cap b = \nil$. This implies conclusion (1). Assume, therefore, that $Q \cap b \neq \nil$. If $Q$ is compressible or $b$-boundary compressible, then $Q = \ob{Q}$ and it is $b$-boundary compressible since $\ob{Q}$ is minimal with respect to $b$-boundary compressions. Performing the $b$-boundary compression creates an essential disc in $N$, contradicting Lemma \ref{Lem: comp arc implies incomp boundary}. Thus, $Q$ is incompressible and $b$-boundary incompressible and $Q \cap b \neq \nil$.

By Proposition \ref{Thm: MSC 2}, we have
\[0 = -2\chi(\ob{Q}) \geq q(\mc{M}_b(a) - 2) + q^*(\mc{M}_b(a^*) - 2) + \mc{M}_b(\boundary \ob{Q}).\]
Each term of the right hand side is non-negative. Hence, each term is 0. If $b$ is separating, this implies that \[ \mc{M}_b(\boundary \ob{Q}) = |b \cap \boundary \ob{Q}| = 0\] and so $\ob{Q}$ is a torus or meridional annulus.

If $q > 0$, then $\mc{M}_b(a) = 2$ and we are done, since this is conclusion (2).

If $q = 0$ we must have $q^* > 0$ and so $\mc{M}_b(a^*) = 2$. If $b$ were separating, we would have $|b \cap a^*| = 2$. Two essential separating curves on a genus two surface that intersect twice are actually parallel, hence $|a \cap b| = 0$, a contradiction. Thus, $L_b$ is a knot. Since all components of $\boundary Q - \boundary \ob{Q}$ are parallel to $a^*$, the 2--component link $L_{a^*}$ is not simple. This is conclusion (3).
\end{proof}

Knots and links obtained by boring split links and unknots using complicated boring arcs are even more special: the dual boring is not complicated.

\begin{theorem}\label{Thm: Dual boring not comp}
Suppose that $L_b \subset S^3$ is a knot or 2-component link obtained from a split link or unknot $L_a$ using a complicated boring arc $\alpha$. Then $L_b$ is not a split link or unknot and $L_b$ has a minimal genus Seifert surface disjoint from the dual boring arc $\beta$. 
\end{theorem}
\begin{proof}
By Corollary \ref{Cor: MSC 2}, $L_b$ is not a split link or unknot. Let $N$ be the exterior of $L_a \cup \alpha$. By Lemma \ref{Lem: comp arc implies incomp boundary}, $N$ is boundary-irreducible. Let $\ob{R}$ be a splitting sphere or unknotting disc for $L_a$ and let $\ob{Q} \leq \ob{R}$ be minimal with respect to $b$-boundary compressions and assume $\ob{Q}$ is connected. Since $\ob{Q}$ is a sphere or disc, 
\[
-2\chi(\ob{Q}) < q(\mc{M}_b(a) - 2) + q^*(\mc{M}_b(a^*) - 2) + \mc{M}_b(\ob{Q}).
\]
Thus, by Proposition \ref{Thm: MSC 2}, for all non-zero classes $y \in H_2(N[b],\boundary N[b])$, with the property that if $b$ is separating then the projection of $\boundary y$ to the first homology of each component of $\boundary N[b]$ is non-zero, there is a conditioned norm-minimizing representative of $\pm y$ disjoint from $\beta$. Choosing $y$ to be a class representable by Seifert surfaces of $L_b$, we see that $L_b$ has a minimal genus Seifert surface disjoint from $\beta$. (Note that reversing the orientation of a Seifert surface of minimal genus in its homology class still gives a Seifert surface of minimal genus in its homology class, so the distinction between $\pm y$ does not matter if, as we do, consider unoriented knots and links.)
\end{proof}

\begin{remark}
The proof of the theorem (stemming from the application of Theorem \ref{Thm: Main Theorem}) actually gives the stronger result that $\beta$ can be isotoped to lie on the branched surface associated to a certain taut sutured manifold hierarchy for the exterior of $L_b$.
\end{remark}

\section{Rational Tangle Replacement}\label{RTR}
Rational tangle replacement is a particularly tractable boring operation. If $L_a$ and $L_b$ are related by rational tangle replacement, then $\mc{M}_b(a) = |a \cap b| = \mc{M}_a(b)$. Furthermore, all the arcs of $b - a$ join a component of $\boundary \eta(a)$ to itself and, together with an arc in $\eta(\boundary \alpha) \subset \boundary_0 N[a]$ form a meridian of $L_a$. Thus, if $\ob{Q} \leq \ob{R}$ and if $\ob{R}$ has empty or meridional boundary, then $\ob{Q}$ does as well. In fact, a closer examination of Lemma \ref{Thm: Constructing Q}gives the following result whose proof we leave as an exercise.

\begin{lemma}\label{Lem: RTR Q}
Suppose that $\ob{R}$ is an essential surface in $N[a]$ with $\ob{R} = Q$ and $\boundary \ob{R}$ isotoped to intersect $b$ minimally. Suppose that $\ob{Q} \leq \ob{R}$. Then one of the following is true:
\begin{enumerate}
\item $\ob{Q} = \ob{R}$ and $Q$ is incompressible and $b$-boundary incompressible.
\item $Q = \ob{Q} = \ob{R}$ and $\alpha$ is isotopic into $Q$
\item $\ob{Q} < \ob{R}$ and $\ob{Q}$ has meridional boundary on some component of $L_a$.
\end{enumerate}
\end{lemma}

We can now give the promised new proofs that unknotting number one knots are prime and that genus is superadditive under band connect sum. The proof of the first of these relies on Theorem \ref{Cor: Non-simple Knot}, and the proof of the second follows almost immediately from Theorem \ref{Thm: Dual boring not comp}. Both of those theorems were consequences of Proposition \ref{Thm: MSC 2}, so we do have a unified sutured manifold theory framework for answering questions about rational tangle replacement.

\begin{theorem}[{\cite{S1}}]\label{Cor: Prime Unknotting 1}
Unknotting number one knots are prime.
\end{theorem}
\begin{proof}
Suppose that $L_a$ is a composite unknotting number one knot. Let $L_b$ be the unknot obtained by changing a crossing of $L_a$. 

\textbf{Case 1:} $\alpha$ is disjoint from an essential meridional annulus for $L_a$.

In this case, the crossing change occurs on a single summand of $L_a$. Such a crossing change can never unknot $L_a$.

\textbf{Case 2:} $L_b$ has an unknotting disc with interior disjoint from $\beta$.

Let $D$ be the 3--ball complementary to the site where the rational tangle replacement occurs so that $(D,L_a \cap B) = (D,L_b \cap B)$ is a 2--tangle. Isotope the unknotting disc so that it intersects $\boundary D$ minimally. It is not difficult to see that this implies that the strands of $L_a \cap D$ are parallel. It is then easy to construct a genus 1 Seifert surface for $L_a$. Since $L_a$ is non-trivial, this is a minimal genus Seifert surface. Since $L_a$ is composite, it must have genus at least 2. Thus, every unknotting disc for $L_b$ has its interior intersected by $\beta$.

\textbf{Case 3:} The interior of every unknotting disc for $L_b$ is intersected by $\beta$.

In this case $\beta$ is a complicated boring arc. By Case 1, we may assume that $\alpha$ is not disjoint from any essential meridional annulus. By Corollary \ref{Cor: Non-simple Knot}, $\mc{M}_b(a) = 2$. Since $b$ is separating, we have $|a \cap b| = 2$. This contradicts the fact that $d = |a \cap b|/2 = 2$.
\end{proof}

\begin{theorem}[{\cite[Theorem 8.4]{S3} and \cite{G4}}]\label{Thm: Genus superadd}
If $K$ is the band sum of knots $K_1$ and $K_2$ using a band $\beta$, then the genus of $K$ is at least the sum of the genera of $K_1$ and $K_2$. Equality holds if and only if the interior of $\beta$ is disjoint from minimal genus Seifert surfaces for $K_1$ and $K_2$.
\end{theorem}
\begin{proof}
If the band sum is actually a connect sum, the result follows from the additivity of genus under connect sum. Otherwise, the band $\beta$ is complicated. Since (by the definition of ``band sum'') $K_1 \cup K_2$ is a split link, by Theorem \ref{Thm: Dual boring not comp}, $K$ has a minimal genus Seifert surface disjoint from cocore of the band. Easy cut and paste arguments, as in \cite[Theorem 8.4]{S3}, finish the proof of the theorem.
\end{proof}

\section{Band sums satisfy the cabling conjecture}\label{Band Sum}

If $K \subset S^3$ is a knot such that the manifold $M_K(m/n)$ obtained by $m/n$ Dehn surgery on $K$ has an essential surface $\wihat{Q}$ not disjoint from the core of the surgery solid torus, then in the exterior of $K$ in $S^3$, then $K$ can be isotoped so that the surface $\wihat{Q} - \inter{\eta}(K)$ is an essential surface with boundary having slope $m/n$ on $K$. We use this observation to prove that band sums satisfy the cabling conjecture. The heart of the matter is contained in the following theorem:

\begin{theorem}\label{Band Attachment Cabling Conj}
Suppose that a knot $L_a$ is obtained from a 2--component link $L_b$ by attaching a band with complicated core. Then $L_a$ is not a cable knot and no non-trivial surgery on $L_a$ produces a reducible 3-manifold.
\end{theorem}

We define $\alpha$, $\beta$, $W$, and $N$ as usual. Let $L$ be the result of pushing $b$ slightly into $N$. Before embarking on the proof we need a preliminary result.

\begin{lemma}\label{Lem: Making Disjt}
Suppose that $\ob{Q} \subset N[a]$ is a properly embedded incompressible surface and that $Q' \subset N$ is obtained by maximally $b$-boundary compressing $Q$. If $Q'$ is disjoint from $b$ then $\ob{Q}$ can be isotoped to be disjoint from $L$.
\end{lemma}
\begin{proof}
Each $b$-boundary compression can be achieved by an isotopy of $\ob{Q}$ that makes it ``dip into'' $W$ without making it intersect $L_a \cup \alpha$. Thus, $Q'$ is obtained from $Q$ in $S^3 - (L_a \cup \alpha)$ by an isotopy of $\ob{Q}$. Hence, if $Q' \cap L = \nil$, then $\ob{Q}$ can be isotoped to be disjoint from $L$.
\end{proof}

\begin{proof}[Proof of Theorem \ref{Band Attachment Cabling Conj}]
Suppose that a surgery on $L_a$ of slope $m/n$ with $(m,n) = 1$ produces a reducible manifold. By \cite{GL}, we know that $n = 1$. Then there is an essential connected planar surface $\ob{Q}$ in the exterior of $L_a$ with boundary of slope $m/1$ on $L_a$. By Lemma \ref{Lem: RTR Q}, there is an isotopy of $\ob{Q}$ so that one of the following occurs:
\begin{enumerate}
\item $Q = \ob{Q} \cap N$ is incompressible and $b$-incompressible
\item $Q = \ob{Q}$ and there is a $b$-boundary compressing disc for $\ob{Q}$ giving rise to an isotopy of $\alpha$ into $\ob{Q}$.
\end{enumerate}

\textbf{Case 1:} $\ob{Q}$ is an annulus and $\alpha$ is isotopic into $\ob{Q}$.

Let $D$ be the $b$-boundary compression giving rise to the isotopy of $\alpha$ into $\ob{Q}$. Performing the $b$-boundary compression creates an essential disc in $N$, contradicting the fact that $\beta$ is complicated. (Lemma \ref{Lem: comp arc implies incomp boundary}). \qed(Case 1).

Let $\gamma = b$ and recall that $(N,\gamma)$ is a taut sutured manifold and that $\gamma$ respects $Q$. If $Q$ is incompressible and $b$-boundary incompressible, Proposition \ref{Thm: MSC 2} implies that
\[
-2\chi(\ob{Q}) \geq 2|\boundary \ob{Q}|
\]
since each component of $\boundary \ob{Q}$ intersects $b$ twice and $\wihat{\gamma} = \nil$. But $\ob{Q}$ is a planar surface and so we would have
\[
-2 + |\boundary \ob{Q}| \geq |\boundary \ob{Q}|
\]
an obvious contradiction. Hence, $Q$ is not incompressible and $b$-boundary incompressible. In fact, it must be incompressible but $b$-boundary compressible. Consequently, $Q = \ob{Q}$ and $\alpha$ is isotopic into $\ob{Q}$. Thus, in light of Case 1, we will be done if we can prove that it is possible to take $\ob{Q}$ to be an annulus.

Let $Q_0 = \ob{Q}$. Each component of $\boundary \ob{Q}$ intersects $b$ exactly twice. Since $\ob{Q}$ is incompressible in $N[a]$, any $b$-boundary compressing disc must join a component of $\boundary \ob{Q}$ to itself (and cross $a \subset \boundary N$). Assume that we have defined $Q_i$ and let $Q_{i+1}$ be obtained from $Q_i$ by a $b$-boundary compression. We obtain a sequence $Q_0, Q_1, \hdots, Q_p$ so that $Q_p$ is $b$-boundary incompressible. (Possibly $Q_p$ is disjoint from $b$.) Each $b$-boundary compression must join a component of $\boundary \ob{Q}$ to itself and must cross $a$, since otherwise we would have a $b$-boundary compression for $\ob{Q}$ in $N[a]$. Thus, each $b$-boundary compression removes a component of $\boundary \ob{Q}$ and replaces it with two components disjoint from $b$. In particular, the boundary of each $Q_i$ intersects $b$ minimally up to isotopy.

\textbf{Case 2:} $Q_p$ is not disjoint from $b$.

Since $Q_p$ is incompressible and $b$-boundary incompressible in $N$ and intersects $b$ minimally, we may apply Theorem \ref{Thm: Main Theorem} to deduce that
\[
-2\chi(Q_p) \geq |\boundary Q_p \cap b|
\]

Since $Q_p$ was obtained from $\ob{Q} = Q$ by $b$-boundary compressing $p$ times, we have
\[
-2\chi(Q_p) = -2(\chi(\ob{Q}) + p) = -2\chi(\ob{Q}) - 2p.
\]

Each boundary compression converted a component of $\boundary \ob{Q}$ that intersected $b$ twice into two components of $\boundary Q_p$ that are disjoint from $b$. Thus, $|\boundary Q_p \cap b| = 2(|\boundary \ob{Q}| - p)$. Consequently,
\[
-2\chi(\ob{Q}) - 2p \geq 2|\boundary \ob{Q}| - 2p
\]
Hence,
\[
-\chi(\ob{Q}) \geq |\boundary \ob{Q}|.
\]
which, as before, is a contradiction to the fact that $\ob{Q}$ is a planar surface. Hence, this case cannot occur.

\textbf{Case 3:} $Q_p$ is disjoint from $b$.

By Lemma \ref{Lem: Making Disjt}, $\ob{Q} = Q$ can be isotoped (relative to its boundary) to be disjoint from $L$. Since $L$ is the unknot in $S^3$, $\ob{Q}$ lies in a solid torus $V$. Performing $m/n$ surgery on $L_a$ in $V$ produces a reducible manifold $V(m/n)$. Since $\beta$ is a complicated boring arc, $L_a$ cannot lie in a 3--ball in $V$. (I.e. the band sum is not a connected sum.) By \cite[Corollary 4.4]{S4}, $L_a$ is cabled and $m/n$ is the slope of the cabling annulus. Hence, it is possible to take $\ob{Q}$ to be an annulus, and we are done.
\end{proof}

We conclude with a simple corollary that also draws on work of Scharlemann.

\begin{corollary}\label{Cor: Band Sums CC}
Suppose that $K$ is the band sum of two knots $K_1$ and $K_2$. Then $K$ satisfies the cabling conjecture and is not a cable knot.
\end{corollary}
\begin{proof}
If the core of the band is not complicated then it intersects a splitting sphere for $K_1 \cup K_2$ exactly once. If it intersects a splitting sphere exactly once, the band sum is a connected sum. By Theorem \ref{Band Attachment Cabling Conj}, we may assume that the band sum is a connect sum. By \cite[Corollary 4.5]{S4}, $K$ satisfies the cabling conjecture. It is well-known that cable knots are prime and so we are done.\end{proof}

\end{document}